\documentclass[reqno,12pt]{amsart}
\usepackage{amsmath, latexsym, amsfonts, amssymb, amsthm, amscd, stmaryrd, color}
\usepackage[utf8]{inputenc}
\usepackage{xcolor}
\usepackage{graphics,epsf,psfrag}
\usepackage{graphicx}
\newcommand{\ind}{\mathbf{1}}

\newcommand{\cL}{{\ensuremath{\mathcal L}} }
\newcommand{\cT}{{\ensuremath{\mathcal T}} }

\newcommand{\bP}{{\ensuremath{\mathbf P}} }

\newcommand{\bX}{{\ensuremath{\mathbf X}} }
\newcommand{\dd}{\mathrm{d}}

\newcommand{\bbE}{{\ensuremath{\mathbb E}} }

\newcommand{\bbP}{{\ensuremath{\mathbb P}} }

\newcommand{\bbR}{{\ensuremath{\mathbb R}} }

\newcommand{\bbZ}{{\ensuremath{\mathbb Z}} }

\newcommand{\mintwo}[2]{\min_{\substack{#1 \\ #2}}}
\newcommand{\lint}{\llbracket}
\newcommand{\rint}{\rrbracket}

\newcommand{\gep}{\varepsilon}       

\newcommand{\gO}{\Omega}
\newcommand{\gl}{\lambda}

\theoremstyle{plain}
\newtheorem{theorem}{Theorem}[section]
\newtheorem{lemma}[theorem]{Lemma}
\newtheorem{proposition}[theorem]{Proposition}
\newtheorem{cor}[theorem]{Corollary}
\theoremstyle{remark}
\newtheorem{rem}[theorem]{Remark}

\numberwithin{equation}{section}

\begin{document}


\title[Mixing time for 1D particle systems]{Mixing time and cutoff for one dimensional particle systems}

\author[H.~Lacoin]{Hubert Lacoin}
\address{IMPA - Estrada Dona Castorina, 110
Rio de Janeiro - 22460-320 - RJ, Brazil}

 \begin{abstract}
  We survey recent results concerning the total-variation mixing time of the simple exclusion process  on the segment (symmetric and asymmetric)  and a continuum analog, the simple random walk on the simplex with an emphasis on cutoff results. A Markov chain is said to exhibit cutoff if on a certain time scale, the distance to equilibrium drops abruptly from $1$ to $0$. We also review a couple of techniques used to obtain these results by exposing and commenting some elements of proof.
 \end{abstract}

  \maketitle

\section{A short introduction to Markov chains}

\subsection{Definition of a Markov chain}

A stochastic process $(X_t)_{t\ge 0}$ indexed by $\bbR_+$ with value in a state-space $\gO$ is said to be a \textit{Markov process} if at each time $t\ge 0$, the distribution the future $(X_{t+u})_{u\ge 0}$, conditioned to that of the past $(X_s)_{s\in [0,t]}$ is only determined by its present state $X_t$. This is equivalent to say that  for every bounded measurable function $F: \gO^{\bbR_+} \to \bbR$ there exists $G: \gO \to \bbR$ such that
\begin{equation}\label{Markov}
 \bbE\left[ F\left[ (X_{t+s})_{s\ge 0} 
 \right] \ | \ (X_{u})_{u\in [0, t]
}\right]= G(X_t).
\end{equation}
The assumption \eqref{Markov} can be interpreted as an \textit{absence of memory} of the process and is called the \textit{Markov Property} (we refer to \cite[Chap. III]{Rowi2000} for an introduction to Markov processes).
\textit{Markov chains} are Markov processes  which are right continuous  for the discrete topology on $\gO$,
 meaning that $(X_t)$ always remains some time in its current state before jumping always from it 
 $$\forall t\ge 0, \quad \inf\{ s, X_{t+s}\ne X_t\}>0.$$
 
 \begin{rem}
The name  Markov chains also (and perhaps more frequently) refers to discrete time Markov processes, that is processes indexed by $\bbZ_+$ rather than $\bbR_+$,  see for instance \cite{Revuz84}. Let us mention that the all the continuous time Markov chains mentioned in this paper are equivalent to discrete time Markov chains - in the sense that they can be obtained by composing a discrete time Markov chain with an homogeneous Poisson process on $\bbR$ - even when the considered state-space is infinite. 
In particular they are  are càd-làg and do not display 
accumulation of jumps (a phenomenon called \textit{explosion} see \cite[Chapter 4]{Strook2014}).
We study these processes in continuous time rather 
than discrete time mostly for practical and aesthetic reasons, 
but the results remain valid for the discrete time version of 
the chains (and the adaptation of the proof from one setup to another is straightforward see for instance 
\cite[Appendix B]{Lac16}).
While some references we mention, such as \cite{Wil04}, 
mention only the discrete time version of the chains, we always transpose the cited 
the cited results in the continuous time setup for a better presentation.
\end{rem}

\subsection{Markov semigroup, generator, invariant measures and reversibility}

\noindent The distribution of a Markov chain $(X_t)_{t\ge 0}$ is determined  by two inputs:
\begin{itemize}
 \item [(A)] The distribution of its initial condition $X_0$, which is a probability distribution on $\gO$, which we denote by $\mu$.
 \item [(B)] The rules of evolution of the future given the present,
 that is, the mapping 
 $\left(\gO^{\bbR_+} \to \bbR\right) \to \left( \gO \to \bbR\right)$,
 that associates $G$ to $F$ in Equation \eqref{Markov}. 
 It can be encoded in an operator acting on functions defined on $\gO$, the \textit{generator} of the Markov chain.
\end{itemize}
Since, in the present paper, we are interested in statements which are valid for every initial distribution $\mu$, when introducing examples of Markov chains, we are going to specify only their generator.

\subsubsection{In finite state-space}
 Let us start by defining the generator of a Markov chain in the simpler case when the state-space is finite (the reader can find in \cite[Chapter 20]{MCMT} a 
 more substancial introduction and proofs 
 of the results mentioned in this section).
An important intermediate step is the definition of Markov semigroup 
$(P_t)_{t\ge 0}$ associated with the Markov chain. It is a sequence of $\gO\times \gO$ 
matrices such that  satisfy the semigroup property  
$P_{s+t}= P_s P_t$ (where matrix multiplication is considered) and such 
that for every $x,y\in \gO$ and $s,t\ge 0$, when $\bbP[X_s=x]>0$, then
\begin{equation}\label{semikov}
\bbP\left[ X_{t+s}=y
 \ | \ X_s=x\right]=P_t(x,y).
\end{equation}
Note that $(P_t)_{t\ge 0}$ jointly with the initial 
distribution fully determines the finite dimensional 
distributions of the process since the iteration of 
\eqref{semikov} yields
\begin{multline}\label{findimarg}
 \bbP\left[X_0=x_0, X_{t_1}=x_1,  X_{t_1+t_2}=x_2, \cdots,  X_{\sum_{i=1}^k t_i}=x_k\right]
 \\=\bbP\left[X_0=x_0\right]  P_{t_1}(x_0,x_1)P_{t_2}(x_1,x_2)\cdots P_{s_k}(t_{k-1},t_k).
\end{multline}
The semigroup property together with our assumption that $(X_t)$ is càd-làg imply that there exists an $\gO\times \gO$ matrix $\cL$ - the generator of the Markov chain - which is such that for all $t\ge 0$ 
\begin{equation*} 
\forall t>0,\quad   P_t= e^{\cL t}:= \sum_{k=1}^{\infty} \frac{s^k}{k!}\cL^k.
\end{equation*}
Note that when we have for $x,y \in \gO$, $x\ne y$
\begin{equation}\begin{cases}\label{defderi}
 \cL(x,y)= \lim_{t\to 0} \frac{1}{t}P_t(x,y),\\
 - \cL(x,x)= \lim_{t\to 0} \frac{1}{t}\left(1-P_t(x,x)\right),
  \end{cases}
\end{equation}
so that $\cL(x,y)$ represent the rate at which our Markov chains jumps from $x$ to $y$ while $-\cL(x,x)$ corresponds to the rate at which the chain jumps away from $x$.
In practice when introducing the generator of a Markov chain, 
we simply write its action (by left multiplication) on $\bbR$ valued functions on $\gO$. That is
\begin{equation*}
 \cL f(x):= \sum_{y\in \gO} \cL(x,y)f(y)=  \sum_{y\in \gO \setminus \{x\}} \cL(x,y)[f(y)-f(x)].
\end{equation*}
We focus on the case of \textit{irreducible} Markov chains, that is, we assume that every state of $\gO$ can be reached from any other states with a finite number of jumps. Formally, for each $x$, $y$ there exists $k\ge 1$ and a sequence $x_0,x_1,\dots, x_k$ with $x_0=x$ and $x_k=y$ such that 
$$ \forall i\in \lint 1, k\rint, \quad \cL(x_{i-1},x_i)>0.$$
This condition immediately implies that $P_s(x,y)>0$ for every $x,y\in \gO$.
If $\cL$ is irreducible,  and $\bbP_{\mu}$ denotes the law of the Markov chain with generator $\cL$ and initial distribution $\mu$, then there exists a unique probability $\pi$ on $\gO$ such that 
$\bbP_{\pi}(X_t=x)=\pi(x)$ for every $\pi$. Such a probability is called the \textit{invariant distribution} of the Markov chain. Considering $\pi$ as a (line) vector on $\gO$  this is equivalent to either of the two relations below
\begin{equation}\label{definvariant}
\begin{cases}
\forall t>0, \quad    \pi P_t=\pi,\\
\pi \cL=0.  
\end{cases}
\end{equation}
The convergence  theorem for irreducible finite state-space Markov chains states (see 
for instance \cite[Theorems 4.9 and 20.1]{MCMT})
that the invariant probability measures $\pi$ is also the limit distribution for 
$X_t$ when $t\to \infty$. More precisely for any probability $\mu$ on $\gO$ we have
\begin{equation}\label{convtoinv}
\lim_{t\to \infty} \bbP_{\mu}(X_t=x)=\pi(x).
\end{equation}
We want to investigate quantitative aspect of this convergence.
For the Markov chains  in this paper, the stationary measure satisfy the so-called  \textit{detailed balance} condition
\begin{equation}\label{detbalance}
\forall x, y \in \gO, \quad \pi(x) \cL(x,y)= \pi(y) \cL(y,x),
\end{equation}
where we use the notation
\begin{equation}\label{intiniterval}
\lint a,b\rint:= [a,b]\cap \bbZ 
\end{equation}
It can be easily checked that  \eqref{detbalance} implies \eqref{definvariant}, but there are irreducible Markov chains for which the stationary probability does not satisfy \eqref{detbalance}  . Markov chains for which the stationary measure satisfies \eqref{detbalance} are called \textit{reversible}.

\subsubsection{When the  state-space is a continuum}

When our state-space is a continuum the above description of the generator as a matrix cannot be used. In that case the semi-group associated to the Markov chain $(P_t)_{t\ge 0}$ is a sequence of probability kernels such that for every bounded measurable function $f$ on $\gO$
 \footnote{Strictly speaking, the relation \eqref{markov2} does not uniquely defines $(P_t)_{t\ge 0}$, since one can modify $P_t(x,\cdot)$ for a set of $x$  which is visited with probability zero    but this is not a relevant issue for our discussion.} every $s$ and $t$ we have
\begin{equation}\label{markov2}
 \bbE\left[ f(X_{t+s}) 
 \ | \ X_{s}\right]= P_t f(X_s) \quad \text{ with } \quad   
  P_t f(x):= \int_{\gO} f(y)P_t(x,\dd y).
 \end{equation}
Informally $P_t(x,A)$ is the probability that $X_{s+t}\in A$ knowing that $X_s=x$.
In analogy with \eqref{findimarg}, the semi-group $(P_t)_{t\ge 0}$ jointly with the initial distribution determines fully the finite dimensional distributions of $(X_t)_{t\ge 0}$.
The generator of the Markov chain  $\cL$ can be defined in analogy with \eqref{defderi} by
\begin{equation}\label{defgess}
 \cL f:=\lim_{t\to 0} \frac{P_t f -f}{t}.
\end{equation}
For a general Markov processes, the limit on the r.h.s.\ in 
\eqref{defgess} may  not exist for every bounded measurable $f$, 
the set of functions for which the limit \eqref{defgess} does exists is called the \textit{domain} 
of the generator.
In this paper however, we are going to consider only Markov chains with uniformly bounded 
jump rates so we won't have to worry about this.
 The conditions for existence and uniqueness of a stationary probability distribution
 and for a convergence such as \eqref{convtoinv} in continuous state-space are very far
 from being as nice as in the finite case (see for instance \cite[Chapter 3]{Rowi2000}). 
 In this survey, we  consider only  chains for which  the stationary measure exists and is unique.
 They also satisfy the continuum counterpart of \eqref{detbalance}, that is to say that the operator
 $\cL$ is self-adjoint in $L_2(\pi)$.


\subsection{Total variation distance and mixing time}
In order to quantify the convergence to equilibrium \eqref{convtoinv}, 
we need a notion of distance on the set $M_1(\gO)$ of probability 
measure on $\gO$, equipped with a $\sigma$-algebra (which is simply the power set $\mathcal P(\gO)$ when $\gO$ is finite). 
We consider the \textit{total-variation} distance,
which quantifies
how well two variables with different distributions can be coupled.
Given $\alpha,\beta\in M_1(\gO)$,  the total-variation distance 
between $\alpha$ and $\beta$ is defined by
\begin{equation*}
 \| \alpha-\beta\|_{\mathrm{TV}}:= \sup_{A\subset \gO}|\alpha(A)-\beta(A)|,
\end{equation*}
where the supremum is taken over measurable sets. The following equivalent characterizations 
of the total-variation distance helps to have a better grasp of 
the notion. It is a sort of $L_1$ distance which measures how well two random variables can be coupled.

\begin{proposition}
 
 If $\gO$ is finite or countable we have 
 \begin{equation*}
   \| \alpha-\beta\|_{\mathrm{TV}}:= \frac{1}{2}\sum_{x\in \gO} |\alpha(x)-\beta(y)|
 \end{equation*}
If $\nu$ is a measure on $\gO$ such that both $\alpha$ and $\beta$ are absolutely continuous w.r.t. $\nu$ 
then
\begin{equation*}
   \| \alpha-\beta\|_{\mathrm{TV}}:= \frac{1}{2}\int_{\gO} \left| \frac{\dd \alpha}{\dd \nu}- \frac{\dd \beta}{\dd \nu} \right| \nu(\dd x).
\end{equation*}
We have 
\begin{equation*}\label{caracoupling}
    \| \alpha-\beta\|_{\mathrm{TV}}:= \mintwo{X_1\sim \alpha}{X_2\sim \beta} \bP[X_1=X_2]
\end{equation*}
where the minimum is taken over the set of all probability distribution $\bP$ on $\gO\times \gO$ which have marginal laws $\alpha$ and $\beta$.
\end{proposition}
The total-variation distance to equilibrium of the Markov chain with generator $\cL$ and stationary measure $\pi$ at time $t$ is given by 
\begin{equation*}
d(t):= \sup_{\mu \in M_1(\gO)} \|\bbP_{\mu}(X_t\in \cdot)-\pi \|_{\mathrm{TV}}.
\end{equation*}
where 
$\bbP_{\mu}$ the law of the Markov chain with generator $\cL$ and initial measure $\mu$. 
A standard coupling argument is sufficient to show that $d(t)$ is non-decreasing as a function of $t$. Given $\gep\in (0,1)$, the mixing time associated to the threshold $\gep$ or $\gep$-mixing time of the Markov chain $X_t$ is given by
\begin{equation*}
T_{\mathrm{mix}}(\gep):= \inf\{t>0  : d(t)\le \gep\} = \sup\{t>0 \ : \ d(t)> \gep\}.
\end{equation*}
It indicates how long it takes, for a Markov chain starting from an arbitrary initial condition, to get close to its equilibrium measure. Note that when $\gO$ is finite and the chain is irreducible, \eqref{convtoinv} guarantees that $\lim_{t\to \infty}d(t)=0$ so that $T_{\mathrm{mix}}(\gep)<\infty$ for all $\gep$.
For chains with continuum state, it is relevant to study the mixing time in the form defined above 
only if there is a unique stationary probability measure 
$\lim_{t\to \infty}d(t)=0$. 

\begin{rem}
 In the case when $d(t)\nrightarrow 0$ 
some relevant variant of the mixing time can be defined by considering 
a restriction on the initial condition, for instance by restricting $x$ to a compact subset of $\Omega$, see e.g.\ \cite{CLL21, GaGh21}.
\end{rem}


\subsection{Organization of the paper}\label{sec:orga}
The main object of this paper is to survey some results
and methods concerning the mixing time of some   Markovian one-dimensional particle systems (with discrete and continuum state-space).
In Section \ref{sec:examples} we introduce these processes.
In Section \ref{sec:results} we expose some results obtained with co-authors in the past decade, and propose a short survey of related research. 
In Section \ref{sec:tec} we review a couple of pivotal ideas,
which first appeared in \cite{Wil04} (in a slightly different form)  
and show how they can be combined to obtain (non-optimal) upper bound on the mixing time.
In Section \ref{sec:tecx}, we discuss the technical 
refinements that are required to improve these bounds into optimal results

\begin{rem}
In both Section \ref{sec:tec} and Section \ref{sec:tecx},
we have made the choice to focus exclusively on upper-bound
estimates for the mixing time. For the theorems presented in this survey
- and in most instances of mixing time problems - 
this is generally thought to be the hardest part of the result.
\end{rem}

\noindent\textbf{Some comments on notation.} 
In the remainder of the paper, we always use the letter $\pi$ 
(with superscripts and subscripts to underline the dependence in
parameters) to denote the equilibrium measure of each of the  
considered Markov chains, so that the meaning of, say, $\pi_N$ 
or $\pi_{N,k}$ will depend on the context. 
When several Markov chains with different initial
distribution are considered, we may use a superscript
to underline the initial distribution (for instance $(X^{\pi}_t)$ denotes a Markov 
chain starting from the stationary distribution). If the initial distribution is a 
Dirac mass $\delta_x$ with $x\in \gO$, we write $X^x_t$ rather than $X^{\delta_x}_t$.

\section{One dimensional particle systems and interface models}
\label{sec:examples}

The Markov chains introduced in this section model the motion of particles in a one dimensional space. In each instance, we do not introduce a single chain but rather a sequence of chains, which are indexed by one or two parameters, which correspond to the size of the system and/or the number of particles.
We want to understand the evolution of the mixing time when these parameters diverge to infinity.

\subsection{The interchange process on the segment}
\label{sec:interchange}

\medskip

\noindent\textit{The symmetric interchange process on the segment}. For $N\ge 2$
 twe let  $\mathcal S_N$ denote the symmetric group, that is.  the set of permutations on $N$ elements. 
For $i\ne j$, we let $\tau_{i,j}$ denote the transposition which exchanges the position of $i$ and $j$.
We define the (symmetric) interchange process on the segment $\lint 1, N\rint$ (recall \eqref{intiniterval}) as the Markov chain on $\mathcal S_N$ with generator 
\begin{equation*}
 \cL^{(N)} f(\sigma):= \frac{1}{2}\sum_{i=1}^{N-1} \left[ f(\sigma\circ \tau_{i,i+1})- f(\sigma) \right].
\end{equation*}
It takes little efforts to check that the Markov chain 
described above is irreducible, 
and that the uniform probability on $\mathcal S_N$ satisfies
the detailed balanced condition \eqref{detbalance}. 
A more intuitive description of the process, which we denote by  $(\sigma_t)$ can be obtained using
Equation \eqref{defderi}:
it jumps away from its current stat with rate $(N-1)/2$ (that is to say, 
the time between consecutive jumps are IID exponential variables of mean $2/(N-1)$) and when it jumps, it chooses uniformly among the permutation obtained by composing on the right with a transposition of the form $\tau_{i,i+1}$ for $i\in \lint 1,N\rint$, or in other words,
it interchanges the value of two randomly chosen consecutive coordinates.

An alternative description is to say that $\sigma_t$ is \textit{updated} with a rate $N-1$ (which is twice the previous rate). At an update time $t$,
one coordinate $i\in \lint 1,N-1\rint$ is chosen uniformly at random, and $\sigma_t$ is re-sampled 
by choosing uniformly at random in the set $\Theta(i,\sigma_{t_-})$, where $\sigma_{t-}$ is
used to denote the left limit at $t$ and  
\begin{equation*}
\Theta(i,\sigma):=\Big\{ \sigma'\in \mathcal S_N \ : \ \forall j\in \lint 1,N\rint\setminus \{i,i+1\}, \sigma'(j)=\sigma(j) \Big\}
=\Big\{ \sigma, \sigma \circ \tau_{i,i+1}\Big\}.
\end{equation*}
Note that with this description, at each update, the value of $\sigma_t$ remains 
unchanged with probability $1/2$ . 
This second description might seem at first sight less natural 
than the first one, but turns out to be more convenient to construct monotone couplings, 
see Section \ref{sec:order}.

\begin{rem}
 The process described above is one of the many examples of random walks on $\mathcal S_N$.
 This family of process has attracted attention, since the origin of the study of mixing times, 
 due to the connection it has 
 with the problem of card shuffling (see \cite[Chapter 8]{MCMT} and  references therein).
 The symmetric interchange process, which we have considered here 
 on the segment can be generalized: the study of the mixing properties
 for the interchange process on an arbitrary graphs has been an active field of research, 
 see for instance \cite{CLR10, Oli13, HS21} and references therein.
\end{rem}


\medskip

\noindent \textit{The biased interchange process.}
We  consider a variant of the process which induce a bias towards more ``ordered'' permutations, that is, 
favors moves which drive the chain ``closer'' to the identity permutation. The set $\Theta(i,\sigma)$ 
is composed of two elements.
We let 
$\sigma^{(i,+)}$ be the  element of $\Theta(i,\sigma)$ such that $\sigma^{(i,+)}(i)<\sigma^{(i,+)}(i+1)$  and
$\sigma^{(i,-)}$ denote the  element of $\Theta(i,\sigma)$ such that $\sigma^{(i,-)}(i)>\sigma^{(i,-)}(i+1)$
(intuitively $\sigma^{(i,+)}$ is the permutation which is more ordered). 
Letting $p\in(1/2,1)$ ($p=1/2$ corresponds 
to the symmetric case considered above, 
the case $p\in(0,1/2)$ is equivalent to $p\in(1/2,1)$ 
after reverting the order of the coordinates) and setting $q:=1-p$, 
we define the generator of the biased interchange process of the segment
\begin{equation*}
 \cL^{(p)}_N f(\sigma):= \sum_{i=1}^{N-1}p
 \left[ f(\sigma^{(i,+)})- f(\sigma) \right]+ 
 q \left[ f(\sigma^{(i,-)})- f(\sigma) \right]
\end{equation*}
The introduction of a bias drastically modifies 
the stationary distribution. We let $D(\sigma)$ denote 
the minimal number of transposition of the type $\tau_{i,i+1}$
which we need to compose to obtain  $\sigma$ - it corresponds to  
the distance between $\sigma$ and the identity permutation 
in the Cayley graph generated by the nearest neighbor
transpositions $(\tau_{i,i+1})^{N-1}_{i=1}$ 
(see \cite[Section 3.4]{Lyonsbook} for an introduction to Cayley graphs).
We have
\begin{equation*}
D(\sigma) = \sum_{1\le i<j \le N}
\ind_{\{\sigma(i)>\sigma(j)\}}.
\end{equation*}
Setting $\gl:=p/q$ ($\gl>1$), the probability measure
$\pi^{(p)}_{N}$ defined by
$$\pi^{(p)}_{N}(\sigma)= \frac{\gl^{-D(\sigma)}}{\sum_{\sigma'\in \mathcal S_N} \gl^{-D(\sigma')}}$$
satisfies the detailed balance condition for $\cL^{(p)}_N$.
As $\gl>1$ the measure $\pi^{(p)}_{N}$ concentrates most of its mass in a small neighborhood 
of the identity (more precisely $D(\sigma)$ is typically of order $N$ under $\pi^{(p)}_{N}$ while it is of 
order $N^2$ under the uniform measure).

\subsection{The exclusion process on the segment}
\label{sec:1dexclusion}

This  Markov chain models the evolution of 
particles diffusing on a segment and subject to \textit{exclusion}: each site can host at most one particle.
Letting $N$ denote the length of the segment. A particle configuration is encoded by a sequence of $0$ and $1$ on the segments, 
ones and zeros indicating respectively the presence/absence of  particle 
on a site.
The space of configurations with a fixed number of particle $k$ is defined by
\begin{equation*}
  \gO_{N,k}:= \left\{ \xi, \lint 1, N \rint \to \{0,1\} \  : \ \sum_{i=1}^N \xi(i)=k\right\}.
\end{equation*}
Given $\xi\in  \gO_{N,k}$ and distinct $i,j\in \lint 1,N\rint$, we set 
$\xi^{(i,j)}=\xi\circ  \tau_{i,j}$
and define the generator of the \textit{Symmetric Simple Exclusion Process} (or SSEP) to be 
\begin{equation*}
 \cL_{N,k}f (\xi):= \frac{1}{2}\sum_{i=1}^{N-1} \left[ f(\xi\circ \tau_{i,i+1})-  f(\xi)\right].
\end{equation*}
An intuitive way to describe the above  Markov chain is to say that each particle particle performs an independent, continuous time nearest neighbor random walk with jump rate
$1/2$ to the left and to the right, but that any jump which would result in either, 
a particle moving out of the segment (that is a jump to the site $0$ or $N+1$) or two
particles occupying the same site (that is  a jump of a particle to an already occupied site)
are canceled (see Figure \ref{lafigure}).
The uniform probability on $\gO_{N,k}$ satisfies the detailed balance condition \eqref{detbalance}.

\medskip

Given $p\in (1/2,1)$ we can also define 
the \textit{Asymmetric Simple Exclusion Process} (or ASEP) 
which is a similar process  on $\gO_{N,k}$, but where the particles perform 
a random walk with respective jump rates $p$ and $q$
to the right and to the left. The corresponding generator is 
\begin{multline}
 \cL^{(p)}_{N,k}f (\xi):= \sum_{i=1}^{N-1} p\ind_{\{\xi(i)>\xi(i+1)\}} \left[ f(\xi\circ \tau_{i,i+1})-  f(\xi)\right]\\
 +\sum_{i=1}^{N-1} q\ind_{\{\xi(i)<\xi(i+1)\}}\left[f(\xi\circ \tau_{i,i+1})-  f(\xi)\right].
\end{multline}
Here also the introduction of the bias yields a modification of the stationary probability. The probability which satisfy the detailed balance condition is given by  (recall $\gl= p/q$)
\begin{equation*}
 \pi^{(p)}_{N,k}(\xi):= \frac{\gl^{-A(\xi)}}{\sum_{\xi'\in \gO_{N,k}} \gl^{-A(\xi')}}
\end{equation*}
where 
$A(\xi):= \sum_{i=1}^N(N-i)\xi(i)-\frac{k(k-1)}{2}$ denotes the (minimal) number of particle 
moves that separates $\xi$ from the configuration $\ind_{\lint N-k+1,N\rint}$ with all particles 
packed to the right of the segment. Note that $A(\xi)$ is typically of order $1$ under  
$\pi^{(p)}_{N,k}$ whereas it is of order $N^2$ under the uniform measure.

\subsection{The corner-flip dynamics}

\label{sec:corner-flip}

We consider a Markov chains that models the motion of interface which are subject only to local moves.
The one-dimensional interface is the graph of a 
one dimensional nearest neighbor path which belongs to the state space
\begin{multline}\label{eq:defxink}
 \Xi_{N,k}:= \{ \zeta, \ \lint 0,N\rint\to \bbZ \ : \ \zeta(0)=0,  \zeta(N)=N-2k, \\
 \forall i\in \lint 0,N-1\rint,  |\zeta(i+1)-\zeta(i)|=1  \}.
\end{multline}
We introduce a Markov chain on $\Xi_{N,k}$ , that only change 
the coordinates of $\zeta$ one at a time.
Given $\zeta\in \Xi_N$ and $i\in \lint 1,N-1\rint$ 
we introduce $\zeta^{(i)}$ to be the element of $\gO_N$
for which only the coordinate at $i$ has been changed (see Figure \ref{lafigure}).

\begin{equation*}\begin{cases}
   \zeta^{(i)}(j):= \zeta(j)  \quad &\text{ if } j\ne i,\\
      \zeta^{(i)}(i):= \zeta(i)-2  \quad &\text{ if } \zeta(i+ 1)= \zeta(i- 1):= \zeta(i)-1,\\
      \zeta^{(i)}(i):= \zeta(i)+2 \quad & \text{ if } \zeta(i+ 1)= \zeta(i- 1):= \zeta(i)+1,\\
        \zeta^{(i)}(i):= \zeta(i) \quad & \text{ if } |\zeta(i+ 1)- \zeta(i- 1)|=2.

                \end{cases}
\end{equation*}
The generator of the symmetric corner flip dynamic is given by
\begin{equation*}
 \mathfrak L_{N,k}:=\frac{1}{2}\sum_{i=1}^{N-1} \left[f(\zeta^{(i)})- f(\zeta) \right].
\end{equation*}
A way to  visualized this dynamics is to say that each ``corner'' displayed by the the graph of
the $\zeta$ is flipped with rate $1/2$.
The uniform measure on $\Xi_{N,k}$ satisfies the detailed balance condition \eqref{detbalance}.

\medskip

 \begin{figure}[ht]
\begin{center}
\leavevmode
\epsfxsize = 8 cm
\epsfbox{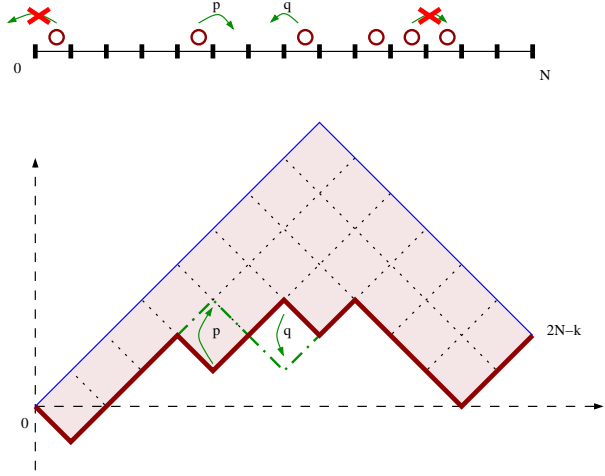}
\end{center}
\caption{\label{lafigure} 
Graphical representation of the exclusion process and of the corner-flip dynamics with $k=6$ and $N=14$.
Particles represented by circle, jump to the right with rate $p$ and to the left with rate $q=1-p$ ($p=1/2$ in the symmetric case), jumps are canceled 
if a particle tries to jump on an already occupied site.
After applying the transformation $h$ given in \eqref{hfromparticle}, 
we obtain the corner flip dynamics, where each downward pointing corner on our interface 
is flipped up with rate $p$ and each upward pointing corner is flipped down with rate $q$.
The quantity $A(\zeta)$ which is the number of up-flips that needs to be performed 
to reach the maximal configuration
$\wedge$ (represented as a thin solid line of the figure) is equal to $22$.
}
\end{figure}

We can also define an asymmetric version of the dynamics which favors flipping the corners in one direction.
Given $\zeta \in \Xi_N$ and $i\in \lint i,i+1\rint$ we define $\zeta^{(i,\pm)}$ to be respectively the ``highest'' and ``lowest'' path in the set $\{\zeta^{(i)},\zeta\}$ (the set 
is possibly a singleton, so that we may have $\zeta^{(i,+)}=   \zeta^{(i,-)}$
)
\begin{equation*}\begin{cases}
   \zeta^{(i,\pm)}(j):= \zeta(j)  \quad &\text{ if } j\ne i,\\
      \zeta^{(i,+)}(i):= \max( \zeta(i),\zeta^{(i)}(i)),\\
    \zeta^{(i,-)}(i):= \min( \zeta(i),\zeta^{(i)}(i))
                \end{cases}
\end{equation*}
       and define the generator of the asymmetric corner-flip dynamics as  
       $$
         \mathfrak L^{(p)}_{N,k}:=\sum_{i=1}^{N-1} p\left[f(\zeta^{(i,+)})- f(\zeta) \right]+ q\left[f(\zeta^{(i,-)})- f(\zeta) \right].$$
        For the asymmetric corner flip, the reversible measure $\pi^{(p)}_{N,k}$ is defined by
        \begin{equation*}
         \pi^{(p)}_{N,k}(\zeta):=\frac{\gl^{-A(\zeta)}}{\sum_{\zeta'\in \Xi_{N,k}}\gl^{-A(\zeta')}}.
        \end{equation*}
where $A(\zeta)$ denote the halved geometric area lying between $\zeta$ and the highest path in $\gO_{N,k}$ defined by
\begin{equation}\label{defwedge}
\wedge(i)= \min\left(i,2(N-k)-i\right),
\end{equation}
that is 
\begin{equation*}
 A(\zeta):=\frac{1}{2}\sum_{i=1}^{N-1} \left(\wedge(i)-\zeta(i)\right).
\end{equation*}

\subsection{Random walk on the simplex}\label{sec:semidis}

Let us finally consider a Markov chain for which the state-space 
is a continuum.
We let $\mathfrak X_N$ denote the $N-1$-dimensional simplex, defined by 
$$ \mathfrak X_N:=\left\{ (x_1,\dots,x_{N-1}) \in \bbR^{N-1} \ : \ 0\le x_1 \le \dots \le x_{N-1}\le N \right\}. $$
We  introduce a dynamic which is a continuum analog of the symmetric exclusion, the coordinates $x_1,\dots,x_{N-1}$ can be thought as the positions of 
$N-1$ particles on the segment $[0,N]$.
The generator of the dynamics, is given by 
\begin{equation*}
 \mathrm L_N f(x)=\sum_{i=1}^{N-1} \int^1_{0} 
 [f(x^{(u,i)})-f(x)]\dd u
\end{equation*}
for $f: \mathfrak X_N\to \bbR$ bounded and measurable, where  $x^{(u,i)}\in \mathfrak X_N$ is defined by
\begin{equation*}\begin{cases}
                  x^{(u,i)}_j:= x_j \text{ for } j\ne i,\\
                  x^{(u,i)}_i:=  ux_{i+1}+(1-u)x_{i-1},                \end{cases}
\end{equation*}
with the convention that $x_0=0$ and $x_N=N$.
In words, at rate one, each coordinate is re-sampled
uniformly in its possible range of value, which is the segment $[x_{i-1},x_{i+1}]$ (see Figure \ref{lotfigure}).

 \begin{figure}[ht]
\begin{center}
\leavevmode
\epsfxsize = 8 cm
\psfrag{x1}{\tiny $x_1$}
\psfrag{x2}{\tiny $x_2$}
\psfrag{xN-1}{\tiny $x_{N-1}$}
\psfrag{popopop}{\tiny $\cdots$}
\epsfbox{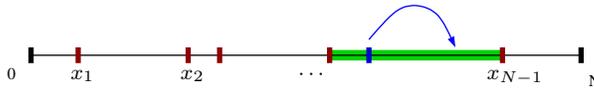}
\end{center}
\caption{\label{lotfigure} 
Graphical representation of the random walk on the simplex ($N=7$).
When the position of a particle is updated ($x_{5}$ on the picture), it is re-sampled 
uniformly in the interval delimited by the neighboring particles (that is $[x_4,x_6]$), with the convention that $x_0=0$ and $x_N=N$
}
\end{figure}

The uniform probability  on  $\mathfrak X_N$, $\pi_N$ defined by
\begin{equation*}
 \pi_{N}(\dd x):= \frac{(N-1)!}{N^{N-1}} \ind_{\{x\in \mathfrak X_N\}} \dd x_1\dots \dd x_{N-1}
\end{equation*}
is stationary for $\mathrm L_N$, and that the generator is self 
adjoint in $L^2(\pi_N)$.

Let us present an explicit construction of the  
 Markov chain $(\bX^x_t)_{t\ge 0}$ with initial distribution $\delta_x$ 
 using auxiliary random variables. 
 Such a construction is referred to as \textit{a graphical construction}
 and turns out to be very convenient to work with 
 (see for instance Section \ref{sec:order} below). It follows the following steps:
\begin{itemize}
 \item [(i)] To each coordinate $i\in \lint 1, N-1\rint$, we associate an independent rate-$1$ Poisson clock process $(\cT^{(i)}_n)_{n\ge 1}$ (the increments of $\cT^{(i)}$ are IID exponential variables of mean $1$) and a sequence of uniform random variables on $[0,1]$,  $(U^{(i)}_n)_{n\ge 1}$.
 \item [(ii)] We set $\bX^{x}(0)=x$. The process is càd-làg and  $(\bX^x(t))_{t\ge 0}$ 
 remains constant on the each of the open intervals of the set
 $(0,\infty)\setminus\{ \cT^{(i)}_n, n\ge 1, i\in \lint 1,N-1\rint \}$ .
 \item [(iii)] At time $t=  \cT^{(i)}_n$, we determine
 $\bX^x(t)$ from $\bX^x(t_-)$ (the left limit at $t$) by setting
$$ X_i(t):=  U^{(i)}_{n}X_{i+1}(t_-)+(1-  U^{(i)}_{n})X_{i-1}(t).$$     
The other coordinates are unchanged: $X_j(t)=X_j(t_-)$ for $j\ne i$.
\end{itemize}
The reader can check that, for any bounded measurable function $f$ and every $x\in \mathfrak X_N$
\begin{equation}\label{droopz}
 \lim_{t\to 0} \frac{ \bbE[f (\bX^x(t))] -  f(x)}{t}= \mathrm L_N f(x).
\end{equation}

\subsection{Correspondences}\label{sec:corre}

The Markov chains presented in Section \ref{sec:interchange}, \ref{sec:1dexclusion} and \ref{sec:corner-flip} are very much related to one another.
Let us first describe the correspondence between particle system and discrete interfaces.
Let us consider $\zeta: \Omega_{N,k}\to \Xi_{N,k}$ defined by 
\begin{equation}\label{hfromparticle}
 h(\xi)(x):= \sum^x_{y=1} (1-2\xi(x)).
\end{equation}
It is immediate to check that $h(\xi)\
\in \Xi_{N,k}$ for every $\xi\in \Omega_{N,k}$ and that $h$ is a bijection (we have $h^{-1}(\zeta)(x)=
\frac{1+\zeta(x-1)-\zeta(x)}{2})$.
Furthermore we have $\mathfrak L^{(p)}_{N,k}\circ h= h \circ\cL^{(p)}_{N,k}$ and as a consequence, if $(\eta_t)_{t\ge 0}$ is a Markov chain with generator 
$\cL^{(p)}_{N,k}$, then its image $h(\eta_t)_{t\ge 0}$ is a Markov chain on $\Xi_{N,k}$ 
with generator $\mathfrak L^{(p)}_{N,k}$  (this is of course also true in the symmetric case, when $p=1/2$).

The corner-flip representation of the exclusion process, can be  convenient for reasoning since it allows for better visual 
representation of 
an order relation which is conserved by the dynamics 
(see Section \ref{sec:order}).

Another useful - although not bijective - correspondence is the one between 
the interchange process and the exclusion process.
Given $k\in \lint 1,N-1\rint$ we define $\xi^{(k)} : \mathcal S_N \to \gO_{N,k}$ as 
\begin{equation*}
 \xi^{(k)}(\sigma)=\ind_{\lint N-k+1,N\rint}\circ \sigma.
\end{equation*}
Since $\cL^{(p)}_{N,k}\circ\xi^{(k)}= \xi^{(k)}\circ \cL^{(p)}_N$, if $(\sigma_t)_{t\ge 0}$ 
is a Markov chain on $\mathcal S_N$ with generator $\cL^{(p)}_N$ 
then  $(\xi^{(k)}(\sigma_t))_{t\ge 0}$ is a Markov chain on $\gO_{N,k}$
with generator  $\cL^{(p)}_{N,k}$. 
The whole sequence of projections
$(\xi^{(k)}(\sigma))^{N-1}_{k=1}$ 
allows to recover $\sigma$ since we have
\begin{equation}\label{thatsall}
\sigma(i)=N-k+1 \quad \Leftrightarrow  \quad \xi^{(k)}(i)-\xi^{(k-1)}(i)=1.
\end{equation}

\section{Review of mixing time results for one dimensional particle systems}
\label{sec:results}

\subsection{The cutoff phenomenon}

Let us now survey a few results  concerning the mixing time of the 
Markov chains introduced in the previous section. 
For all of these processes, an asymptotic equivalent 
to the mixing time  $T^{(N)}_{\mathrm{mix}}(\gep)$ (or $T^{(N,K)}_{\mathrm{mix}}(\gep)$ 
we have several parameters) 
is obtained in the limit when the parameter $N$ (or $N$ and $k$) tends to infinity.
A striking common feature of all these results 
is that in the asymptotic equivalent of  $T^{(N)}_{\mathrm{mix}}(\gep)$, 
there is no dependence in $\gep$ ($T^{(N)}_{\mathrm{mix}}(\gep)$ depends on $\gep$ 
but this dependence only appear in higher order terms).
In particular we have for any $\gep>0$
\begin{equation*}
 \lim_{N\to \infty} \frac{T^{(N)}_{\mathrm{mix}}(\gep)}{T^{(N)}_{\mathrm{mix}}(1-\gep)}=1.
\end{equation*}
This mean that on a certain time scale, 
the distance to equilibrium drops abruptly from $1$ to $0$.
This phenomenon is known as cutoff and is believed to hold for a 
wide class of Markov chain (we refer to \cite{DiaconisPNAS} 
and \cite[Chapter 18]{MCMT} and references therein 
for an historical introduction to cutoff).
Cutoff is most of the time delicate to prove. 
For many Markov chains, while a short argument allow to identify the mixing time up to a constant multiplicative factor (cf. Section \ref{sec:tec}), much more effort is usually needed to obtain asymptotically  matching upper 
and lower bounds.

\subsection{Mixing time results}

\noindent\textit{The SSEP and the Interchange Process.}
 We group the results concerning the exclusion 
 process and the interchange process as their proofs 
 share a lot of common ideas. We start with the symmetric case.
The lower bounds in the result below  have been proved by Wilson in \cite{Wil04} while the upper bounds have  been obtained by the author in  \cite{Lac16}.

\begin{theorem}\label{thm:mixsep}
For any sequence $k_N$ such that $\lim \min(k_N,N-k_N)=\infty)$ the mixing time of 
symmetric exclusion process on $\lint 1,N \rint$ with $k_N$ particle satisfies, for any $\gep\in (0,1)$

\begin{equation}\label{mixssep}
\lim_{N\to \infty} \frac{T^{\mathrm{SSEP},(N,k_N)}_{\mathrm{mix}}(\gep)}{N^2 \log\left[ \min(k_N,N-k_N)\right]}=\frac{1}{\pi^2}
\end{equation}
  For the interchange process on  on $\lint 1,N \rint$  we have for any $\gep\in (0,1)$
 \begin{equation*}
  \lim_{N\to \infty} \frac{T^{\mathrm{IP},(N)}_{\mathrm{mix}}(\gep)}{N^2 \log N}=\frac{1}{\pi^2}
 \end{equation*}
\end{theorem}
\noindent In view of the correspondence discussed in Section \ref{sec:corre}, the first part of the result, that is  \eqref{mixssep}, is also valid for the corner-flip dynamics introduced in Section \ref{sec:corner-flip}.

\medskip

\noindent\textit{The random walk on the simplex.}
The process is in a sense very similar to the the simple 
exclusion process with a positive density of particle. 
However the methods developed in \cite{Lac16, Wil04} - and more generally, 
many of the techniques concerning upper bounds for mixing time - rely on the fact
that the state-space is discrete.
 The following, proved in \cite{CLL},
 is one of the few cutoff results that 
 have been proved for a Markov chain evolving in a continuum
 (see \cite{HoJi17} for another example).

\begin{theorem}\label{thm:mixasep}
 For the random walk on the simplex $\mathfrak X_N$ we have for any $\gep\in(0,1)$
 \begin{equation*}
  \lim_{N\to \infty} \frac{T^{\mathrm{RW},(N)}_{\mathrm{mix}}(\gep)}{N^2 \log N}=\frac{1}{\pi^2}
 \end{equation*}
\end{theorem}

Upper and lower bounds of the right order - that is $N^2 \log N$ - but without the right constant factor 
has been proved prior to the above theorem in \cite{RW05} (see also \cite{RW05b} for similar results in a periodic setting).

\noindent\textit{The ASEP and the  Biased Interchange Process.}
The introduction of a bias has the effect of making the system mix considerably faster: 
a time of order $N$ is required for mixing  instead of $N^{2+o(1)}$ in the symmetric 
case (this was proved in was proved in \cite{Benjamini}).
In \cite{LL19} jointly with C.\ Labbé, we were able to identify the sharp 
asymptotic of the mixing time, proving cutoff both for the ASEP and for the biased interchange process.

\begin{theorem}

For any $p\in (1/2,1]$ and  any sequence $(k_N)$ such that 
$$\forall N\ge 2, k_N\in \lint 1,N-1\rint \quad \text{ and } \quad 
\lim_{N\to \infty}\frac{k_N}{N}=\alpha\in[0,1],$$
the mixing time of asymmetric exclusion process with $k_N$ particle satisfies, for  every $\gep\in(0,1)$
\begin{equation*}
\lim_{N\to \infty} \frac{T^{\mathrm{ASEP},(p,N,k_N)}_{\mathrm{mix}}(\gep)}{N}=\frac{(\sqrt{\alpha}+\sqrt{1-\alpha})^2}{2p-1}.
\end{equation*}
 For the biased interchange process on segment  we have for every $\gep\in(0,1)$
 \begin{equation*}
  \lim_{N\to \infty} \frac{T^{\mathrm{BIP},(p,N)}_{\mathrm{mix}}(\gep)}{N}=\frac{2}{2p-1}
 \end{equation*}

\end{theorem}
Note that the expression for the mixing time in the above result diverges when $p$ tends to $1/2$. In \cite{LePe16,LL20} the crossover regime between the symmetric and asymmetric case is investigated.
The right order of magnitude for the mixing time is established \cite{LePe16} while \cite{LL20} proves cutoff results.

\subsection{Review of related works}

\noindent \textit{Cutoff window and cut of profile.} The results 
above concern the first order asymptotics of the mixing time. However, one can aim for results with a finer precision. For instance one can try to estimate  the order of magnitude of 
$T^{(N)}_{\mathrm{mix}}(\gep)-T^{(N)}_{\mathrm{mix}}(1-\gep)$ (say for a fixed $\gep\in(0,1/2)$, this could theoretically depend on the value of $\gep$, but in practice it does not for most chains), a quantity called the width of  \textit{the cutoff window}.
One can further refine the picture and look for the limit of the distance to equilibrium $d^{(N)}(t)$ after
recentering the picture on  $t=T^{(N)}_{\mathrm{mix}}(1/2)$ and rescaling time by the cutoff window width. 
This is called \textit{the cutoff profile}, and is the finest degree 
of description of convergence to equilibrium.
For the SSEP on   the circle - which is the closest cousin
for the exclusion SSEP on the segment - the cutoff window - of 
order $N^2$ - and profile have been identified in \cite{Lac162, Lac17}.
In the asymmetric case, the cutoff window - of order $N^{1/3}$ - 
and  profile have been identified in \cite{BuNe20} 
(in the case where the density of particle is positive).

\medskip

\noindent \textit{The exclusion process with open boundary condition.}
We have considered above dynamics where the number of particle is conserved. It is possible to consider the case of open boudaries, where particle can enter and exit the segment on the left and on the right.
In that case, the equilibrium  and dynamical behavior of the system depends a lot on the value chosen for the exit and entrance rate of the particle at the left and right boundary. Mixing time results for the 
exclusion of the segment with a variety of boundary condition are proved 
in \cite{GNS2020}, where several open questions and conjectures are
also displayed. One of these conjecture is solved in  \cite{Schmid21}, 
where it shown that  in the maximal current phase, for the  totally
asymmetric exclusion process (TASEP) 
the mixing time in that case is of order $N^{3/2}$. 
A similar result is predicted to hold for the asymmetric 
exclusion process on the circle, and the corresponding 
lower bound on the mixing time can be deduced from the results 
in \cite{BaLi18}.

\medskip

\noindent \textit{The exclusion process in a random environment.}
Another variant of the process has been considered 
where the bias that each particles feels depends on the site on which it lies. 
That is to say $p$ varies with $i$. In \cite{Schmid19, LY21} the case of  IID random biases
has been considered. This is a multi-particle version of the the classical Random Walk in a Random Environment
(RWRE) (see e.g.\ \cite{Sol75, KKS75} for seminal references and \cite{GK13} for a study of mixing time of RWRE).
The works \cite{LY21, Schmid19} show that the presence of inhomogeneities in 
the environment  can slow down the convergence to equilibrium. 

\medskip

\noindent \textit{The exclusion process in higher dimension.}
The symmetric  exclusion process on a higher dimensional rectangle or torus has also been investigated. Proving result beyond dimension one turns out to be more difficult since monotonicity (in the sense of Section \ref{sec:order}), which is a tool of crucial importance, cannot be used.
It has been shown in \cite{Morris06} that the exclusion process in that case continues with a mixing time of order $N^2 \log k$ (see also \cite{Qua92, Yau97} for earlier functional inequalities which implies that the mixing time is of order $N^2 \log N$ when there is a density of particles).

\medskip

\noindent \textit{More general interfaces}
The mixing of one dimensional interfaces have been studied very much beyond  the case of 
the corner-flip dynamics.
In \cite{CMT08, CLMST12, CLMST14, LaTe15, LY20, Yang21}, the case of interfaces 
interacting with a substrate has been considered. The references  \cite{CMT12, CLMST14, GRSP2020, CMST} 
investigate the mixing time of higher  dimensional interfaces.
In \cite{CLL21}, interfaces with real valued 
height functions are considered beyond the case of the random walk 
on the simplex. Let us finally mention \cite{GaGh21} which proves 
cutoff for  Gaussian interfaces (the lattice free field) in arbitrary dimension.

\section{Presentation of a few technical tools used to prove these results}
\label{sec:tec}
We review of few key ingredients used  
in the proof of the results presented in the previous section.
More precisely, to illustrate these techniques, we present a proof of non-optimal results concerning the
mixing time  of the simple exclusion process on the segment
(symmetric and asymmetric), or rather, to its corner-flip representation.
Although the presentation slightly differs, 
the argument found below is in spirit very similar to the one found in 
\cite[Section 3]{Wil04}. The reasoning can be applied 
without much change 
to the interchange process (see Remark \ref{andtheinter}) but for clarity and conciseness we limit the exposition of details to the case of the 
exclusion process.
We discuss in Section \ref{sec:tecx} how additional ideas are needed to improve on 
this non-optimal result.

\subsection{Order preservation}\label{sec:order}

Let $\le$ be  a partial order relation on $\gO$.
Given  $\alpha,\beta \in M_1(\gO)$ we say 
$\alpha$ is stochastically dominated by $\beta$ (for the order $\le$)
and write  $\alpha \preccurlyeq \beta$ if one can construct 
 - on the same probability space - a pair of $\gO$-valued variables 
$Z_{\alpha}$ and $Z_{\beta}$ with respective distributions $\alpha$ and $\beta$ such that
such that  we have $Z_{\alpha}\le Z_{\beta}$ with probability one.
The  Markov chain with generator $\cL$ is said to be \textit{order preserving} or \textit{attractive} if 
its semigroup preserves stochastic ordering, that is to say that for any $t>0$
$$\alpha \preccurlyeq \beta \Rightarrow  \alpha P_t  \preccurlyeq \beta P_t.$$
An equivalent way of saying this is to say that the dynamic is order preserving if
if for any  $x,y \in \gO$ 
such that $x\preccurlyeq y$, 
one can couple two Markov chains $(X^x_t)_{t\ge 0}$ and 
$(X^y_t)_{t\ge 0}$ with  respective initial condition $x$ and $y$
in such a way that
\begin{equation*}
 \forall t\ge 0, \quad X^x_t \le X^y_t.
\end{equation*}

\noindent \textbf{Order preservation for the corner-flip dynamics}.
  
 We define $\le$ on $\Xi_{N,k}$ to simply be the coordinate-wise order,
 that is 
\begin{equation}\label{coodorder}
 \zeta\le \zeta' \quad \Leftrightarrow \quad \forall i\in \lint 1, 2N-1\rint, \quad \zeta(i)\le \zeta'(i).
\end{equation}
To show that the corner-flip dynamic on $\Xi_{N,k}$ is order-preserving,
we use a construction which is similar 
to that presented above Equation \eqref{droopz},
using clock processes $(\cT^{(i)}_n)_{i\in \lint 1,N-1\rint, n\ge 0}$ 
(independent Poisson processes with mean one inter-arrival law) 
and  accessory variables  $(U^{(i)}_n)_{i\in \lint 1,N-1\rint, n\ge 0}$
which are IDD uniform variable on the interval $[0,1)$.
The clock processes $(\cT^{(i)})_{n\ge 0}$ determines when the update of the coordinate $i$ are performed, and the variables $U^{(i)}_n$ are used to determine whether the corner should be flipped up or down.
Given $\zeta\in \Xi_{N,k}$ we construct $(h^{\zeta}_t)$ as the unique 
càd-làg process which satisfy:
\begin{itemize}
 \item [(i)] $h^{\zeta}_0=\zeta$,
 \item [(ii)] $(h^{\zeta}_t)_{t\ge 0}$ remains constant on intervals of 
 $\bbR_+\setminus (\cT^{(i)}_n)_{i\in \lint 1,N-1\rint, n\ge 0}$.
 \item [(iii)] If $t=\cT^{(i)}_n$ and $h^{\zeta}_{t_-}=\xi$ then 
 \begin{itemize}
 \item[(A)] Then if $U^{(i)}_n \ge \in [1-p,1)$ set $h_t=\xi^{(i,+)}$  
 \item[(B)] Then if $U^{(i)}_n \ge \in [0,1-p)$ set $h_t=\xi^{(i,-)}$.
 \end{itemize}
\end{itemize}
Since the set $\{\cT^{(i)}_n\}_{i\in \lint 1,N-1\rint, n\ge 0}$ display no accumulation point, $(h^{\zeta}_t)$ can be constructed by performing the updates sequentially.
We can use this construction (using the same $\cT$ and $U$) to obtain a collection of processes 
$(h^{\zeta}_t)$, $\zeta\in \Xi_{N,k}$ constructed on the same probability which is such that 
\begin{equation}\label{orderpreservz}
 \zeta\le \zeta' \quad \Longrightarrow \quad \forall t\ge 0,\quad  h^{\zeta}_t\le h^{\zeta'}_t.
\end{equation}
The validity of \eqref{orderpreservz} follows from  
the fact that each update is order preserving, 
which holds true because for any fixed $i$ 
the applications $\zeta \mapsto \zeta^{(i,\pm)}$ are order preserving.
A coupling such as the one presented above, where  chains 
starting from all initial condition are constructed on a common probability space, 
is called a \textit{grand coupling}.
This type of construction using auxiliary variable is called the 
\textit{graphical construction} and is quite common for 
interacting particle systems or spin systems 
(see for instance \cite[Chap III.6]{Liggett}).

\begin{rem}
For the interchange process, we can use the order which corresponds to \eqref{coodorder} after applying the correspondences of Section \ref{sec:corre}, and a similar construction allows to obtain a monotone grand coupling.
The analog construction also provide a monotone grand coupling
for the random walk on the simplex.
\end{rem}

\subsection{Connection with the discrete heat equation}

Let us  expose  first  how  the evolution of the mean of simple 
observables - the height function in the symmetric case,  the
exponential of the height in 
the asymmetric case - can be described by a simple system of linear equation.

\subsubsection{In the symmetric case}
Given  $\zeta\in\Xi_{N,k}$  we define   $u^{\zeta}(t,\cdot)$, to be the recentered mean height of the interface at time $t$ for the corner-flip dynamic with initial condition $\zeta$
\begin{equation}\label{defu}
  u^{\zeta}(t,i):=  \bbE\left[ h^{\zeta}_t(i)\right] 
  \end{equation}
  For a real valued function $f$ defined on $\lint 1,N-1\rint$, we define 
  $\Delta_{D} f$  ($\Delta_D$ being the discrete Laplace operator
  with Dirichlet boundary condition) by 
 \begin{equation}\label{dirilap}
  \Delta_{D} f(i):=f(i+1)+f(i-1)-2f(i) \quad  \text{ for } i\in \lint 1,N-1\rint.
 \end{equation}
 with the convention that $f(0)=0$
  and 
  $f(N)=N-2k$.
The function $u^{\zeta}$ is the unique solution of the following 
system of differential equations, that can be considered as a partial differential equation where 
the space variable is discrete ($\Delta_D$ acts on the second variable) 
\begin{equation}\label{disdir}
 \partial_t u(t,i)= \frac 1 2  \Delta_D u(t,i), \quad   \forall i\in \lint 1,N-1\rint. 
\end{equation}
Setting  $U^{(i)}(\zeta):= \zeta(i)$, Equation \eqref{disdir} is deduced from the identity  (that can be checked from the definition of the generator)
$$\mathfrak L_{N,k} U^{(i)}(\zeta,i)= \tfrac 1 2 \Delta_D \zeta(i).
$$
 More precisely \eqref{disdir} 
is obtained by  combining \eqref{defgess}, 
the Markov property, the identity , and the fact that, $\Delta_D$ being an affine 
transformation, it commutes with the expectation, as follows 
\begin{equation}\label{follozing}
  \partial_t u(t,i)= \bbE\left[  \mathfrak L_{N,k}U^{(i)}(h^{\zeta}_t)  \right]=\bbE\left[\Delta_D h^{\zeta}_t (i)\right]=
 \Delta_D \left(\bbE\left[ h^{\zeta}_t\right]\right) (i)=  \Delta_D u(t,i).
\end{equation}
The fact that $u^{\zeta}$ does not have zero boundary condition 
is not a problem since in computations, we  
consider the difference $u^{\zeta}-u^{\zeta'}$
which displays zero boundary condition.
The Dirichlet Laplacian with zero boundary condition $\Delta^{(0)}_D$ is a linear operator that can easily be 
diagonalized.
The family  $(\overline\sin^{(j)})_{j=1}^{N-1}$ defined by
$\overline\sin^{(j)}(i):=\sin \big(\tfrac{ij\pi}{N}\big)$
forms a base of eigenvectors of $\Delta^{(0)}_D$ in $\bbR^{N}$
and we have 
\begin{equation}\label{valiouz}
\Delta^{(0)}_D\overline\sin^{(j)}=-2\gamma^{(j)}_N\overline\sin^{(j)} \quad
\text{ where } \quad \gamma^{(j)}_N=1-\cos\big( \tfrac{j\pi}{N} \big).
\end{equation}
 Using Parceval's inequality we obtain the following contractive 
 estimates which we use to bound the mixing time.

\begin{lemma}\label{contrax}
 If $u: [0,\infty) \times \lint 1,N-1\rint$ satisfies  $\partial_t u= \Delta^{(0)}_D u$, then 
 we have for any $t\ge 0$
 \begin{equation*}
  \sum_{i=1}^{N-1} u(t,i)^2 \le e^{-2\gamma^{(N)}_1 t}   \sum_{i=1}^{N-1} u(0,i)^2.
 \end{equation*}
 
\end{lemma}

\subsubsection{In the asymmetric case}

When $p\ne 1/2$, the quantity $\mathfrak L^{(p)}_{N,k} U^{(i)}(\zeta)$ 
cannot be expressed as a linear combination of 
  $U^{(j)}(\zeta)$, $j\in \lint 1,N\rint$ 
  so that there is no way to recover 
 a linear system analogous to \eqref{disdir} for the averaged 
 heights.
 
 \medskip
 
However we can obtain something similar  for 
the evolution of an averaged quantity related the heights. 
The key idea which can be traced by to \cite{GART88} 
(where it is used to derive hydrodynamic limits)  
is to apply the so-called discrete Cole-Hopf transform.
We consider exponentials of  heights rather the than heights themselves.
Recalling that 
 $\gl=p/q$, we define 
\begin{equation*}
 V(\zeta,i):= \gl^{\frac{1}{2}\zeta(i)} \quad  \text{ and } \quad 
  v^{\zeta}(t,i):=\bbE^{(p)}\left[ V(h^{\zeta}_t)(i)\right].
\end{equation*}
Setting $\varrho:= (\sqrt{p}-\sqrt{q})^2$, 
 it can be checked from the definition of the generator  that for every $\zeta$ and $i\in \lint 1, N-1\rint$ we have
 \begin{equation}\label{idid}
 \mathfrak L^{(p)}_{N,k} V(\zeta,i)=  \sqrt{pq}\Delta_D V (\zeta,i) - \varrho
  V(\zeta,i)
  \end{equation}
where this time $\Delta_D$ denotes the Dirichlet Laplacian defined as in \eqref{dirilap} 
but with boundary condition $f(0)=1$ and $f(N)=\gl^{\frac{N}{2}-k}$ 
(we refer to \cite[Section 3.3]{LL19} for details on the computation leading to \eqref{idid}).
In \eqref{idid} note that $\mathfrak L^{(p)}_{N,k} $ acts on the first 
coordinate while $\Delta_D$ acts on the second one.
As in  \eqref{follozing},  we obtain from \eqref{idid} that $v^\zeta$ satisfies
 \begin{equation}\label{leqasymme}
\partial_t v(t,i):= \left(\sqrt{pq}\Delta_D -\varrho\right)  v(t,i), \quad \forall  i\in \lint 1, N-1\rint,
\end{equation}
Again, the non-zero boundary condition  for $\Delta_D$ 
here is of no importance since in practice we are going to consider the difference 
 $v^{\zeta}-v^{\zeta'}$. As in the symmetric case the diagonalization of the operator with $0$ boundary condition
 $\Delta^{(0)}_D$ 
 yields the following estimate.

 \begin{lemma}\label{contrax2}
 If $v$ satisfies $\partial_t v= \sqrt{pq}\Delta^{(0)}_D v-\varrho v  $ 
 then we have
 \begin{equation*}
  \sum_{i=1}^{N-1} v(t,i)^2 \le e^{-2\left(\gamma^{(N)}_1+\varrho\right)t}   \sum_{i=1}^{N-1} v(0,i)^2.
 \end{equation*}

\end{lemma}

\subsection{Using the heat equations to obtain bounds on the mixing time}
\label{proofsec}
Let $(h^{(1)}_t)$ and $(h^{(2)}_t)$ be two \textit{ordered}
corner flip dynamics, that is, such that  $h^{(1)}_t\le h^{(2)}_t$ for all $t$.
Using only Lemmas \ref{contrax},\ref{contrax2} and order preservation we can  control the coupling time 
of $(h^{(1)}_t)$ and $(h^{(2)}_t)$  defined by 
\begin{equation}\label{deftau}
 \tau:= \inf\{t>0 \ : \ h^{(1)}_t=   h^{(2)}_t\}.
\end{equation}

\begin{proposition}\label{monocouple}
 If  $(h^{(1)}_t)$ and $(h^{(2)}_t)$ be two \textit{ordered} symmetric 
corner flip dynamics then for any $t>0$  we have 
\begin{equation*}
 \bbP\left[ \tau>t \right]\le k (N-1) e^{-\gamma^{(N)}_1 t}.
\end{equation*}
  If  $(h^{(1)}_t)$ and $(h^{(2)}_t)$ be two \textit{ordered} asymmetric  
corner flip dynamics with parameter $p$ we have 
\begin{equation*}
 \bbP\left[ \tau>t \right]\le k (N-1) \gl^{N/2-1} e^{-\varrho t}.
\end{equation*}

\end{proposition}
\noindent From these coupling estimates we can derive upper estimates on the mixing time 
\begin{cor}\label{firstesti}
 We have 
 \begin{equation*}\begin{split}
 T^{\mathrm{SSEP},(N,k_N)}_{\mathrm{mix}} &\le \frac{1}{\gamma^{(N)}_1}\log\left( \frac{2k(N-1)}{\gep} \right),\\
 T^{\mathrm{ASEP},(p,N,k_N)}_{\mathrm{mix}} &\le \frac{1}{\varrho}\left[ \left(\frac{N}{2}-1\right)\log \gl + \log\left( \frac{ 2k(N-1)}{\gep} \right) \right]
 \end{split}
 \end{equation*}
\end{cor}
\begin{rem}\label{dizcuz}
Replacing $\gamma^{(N)}_1$ by an asymptotic equivalent ($\frac{\pi^2}{2N^2}$) we find that the upper bound 
on the SSEP mixing time is  $\frac{2N^2}{\pi^2} \left(\log N +\log k\right) (1+o(1))$
which is in the best case, a factor $4$ always from the estimate given in Theorem \ref{thm:mixsep},
For the ASEP, our upper bound is asymptotically equivalent to  
$\frac{\log \gl}{2\varrho} N$.
Since we have for every $p\in(1/2,1)$
$$ \frac{\log \gl}{2\varrho}> \frac{2}{2p-1}=\max_{\alpha\in [0,1]} \frac{(\sqrt{\alpha}+\sqrt{1-\alpha})^2}{2p-1},$$
in this case again the estimate is not sharp.
The reason why the bounds in Corollary \ref{firstesti} are not sharp further discussed in Section \ref{sec:tecx}.
\end{rem}

\begin{proof}[Proof of Corollary \ref{firstesti}]
Using the correspondence of Section \ref{sec:corre}, we can reason with the corner flip dynamics since it has the same mixing time.
 In order to prove an upper bound on the mixing time, one must bound from above the distance between $\bbP[h^{\zeta}_t\in \cdot]$ and the stationary measure $\pi$ for an arbitrary $\zeta\in \Xi_{N,k}$.
 In order to transform this into a coupling problem, note that $\pi=\bbP[h^{\pi}_t\in \cdot]$ where, with some abuse of notation, we let 
 $h^{\pi}_t$ denote a Markov chain with initial condition $\pi$.
 \medskip
 
 Let us consider now three different dynamics, $h^{\zeta}_t$, 
 $h^{\wedge}_t$ and  $h^{\pi}_t$, with respective initial 
 condition $\zeta$, $\wedge$ (defined in \eqref{defwedge}) 
 and stationary. They are constructed on the same probability space and coupled in such a way that for all $t\ge 0$ (Section \ref{sec:order} gives such a coupling)
 \begin{equation*}
  h^{\pi}_t \le h^{\wedge}_t \quad     \text{ and }  \quad h^{\zeta}_t\le h^{\wedge}_t.
 \end{equation*}
Using \eqref{caracoupling}  stationarity and union bound, we have
 \begin{equation*}  \| \bbP[h^{\zeta}_t\in \cdot] -\pi\|_{\mathrm{TV}}
  \le \bbP[ h^{\zeta}_t \ne h^{\pi}_t ]
  \le \bbP[ h^{\zeta}_t \ne h^{\wedge}_t ] +  \bbP[ h^{\pi}_t \ne h^{\wedge}_t ]= \bbP[ \tau_1>t]+ \bbP[ \tau_2>t],
  \end{equation*}
 where we have set 
 \begin{equation}\label{tau1tau2}
  \tau_1:= \inf\{ t \ : h^{\zeta}_t \ne h^{\wedge}_t \} \quad \text{ and } \quad 
    \tau_2:= \inf\{ t \ : h^{\pi}_t \ne h^{\wedge}_t \}.
 \end{equation}
The tail distribution of $\tau_1$ and $\tau_2$
can be estimated using Proposition \ref{monocouple} and we obtain (let us now for the first time highlight the difference in $p$)
\begin{equation}\label{ztimat}
 \begin{cases}
  \| \bbP[h^{\zeta}_t\in \cdot] -\pi\|_{\mathrm{TV}}\le 2(N-1)k e^{-\gamma^{(N)}_1 t}   \text{ in the symmetric case},\\
    \| \bbP[h^{\zeta}_t\in \cdot] -\pi\|_{\mathrm{TV}}\le 2(N-1)k \gl^{N/2-1} e^{-\varrho t}   \text{ in the asymmetric case}.
 \end{cases}
\end{equation}
The reader can then check that the value of $t$ which makes the r.h.s.\ in  \eqref{ztimat} equal to $\gep$ is 
the claimed upper bound on the mixing time.
 \end{proof}

 \begin{proof}[Proof of Proposition \ref{monocouple}]
  Let us start with the symmetric case. We set 
  $$h^{(1,2)}_t(i):= h^{(2)}_t(i)-h^{(1)}_t(i) \quad  \text{ and  }  \quad
  u^{(1,2)}(t,i):= \bbE\left[(h^{(2)}_t-h^{(1)}_t)(i)\right].$$ Since $h^{(1)}_t\le h^{(2)}_t,$ we have $h^{(1,2)}_t(i)\ge 0$  for all $i$, 
and if $h^{(1)}_t\ne h^{(2)}_t$ the inequality must be strict for at least one value of $i$.
  Since the minimal discrepancy between two values of $\zeta(i)$
  is $2$, this implies that 
  \begin{equation}\label{zrunk}
   \bbP[\tau>t]= \bbP\left[ h^{(1)}_t\ne h^{(2)}_t\right] 
   = \bbP\left[ \sum_{i=1}^{N-1} (h^{(1,2)}_t(i)\ge 2\right] \le \frac{1}{2}\sum_{n=1}^{N-1}u^{(1,2)}(t,i).
  \end{equation}
Combining Cauchy-Schwarz with Lemma Lemma \ref{contrax}  
- from \eqref{disdir} we know that $u^{(1,2)}$ satisfies the assumption - we have
\begin{equation*}
 \sum_{i=1}^{N-1} u^{(1,2)}(t,i) 
 \le \left( (N-1)   \sum_{i=1}^{N-1} u^{(1,2)}(t,i)^2 \right)^{1/2}
  \le e^{-\gamma^{(N)}_1 t} \left( (N-1)   \sum_{i=1}^{N-1} u^{(1,2)}(0,i)^2 \right)^{1/2} 
\end{equation*}
and we can conclude using the fact  $u^{(1,2)}(0,i)\le 2k$ since $2k$ is
a bound for the maximal height difference between two elements in $\Xi_{N,k}$.

\medskip

\noindent For the asymmetric case we apply the reasoning  to the  exponential of the heights
\begin{equation}\label{defw12}
W(\zeta):= \sum_{i=1}^{N-1} V(\zeta,i) 
\quad \text{ and } \quad W^{(1,2)}_t:=  W(h^{(2)}_t)- W(h^{(1)}_t)
\end{equation}
Note that since $\zeta(i)\ge -k$ for all $\zeta$ and $i$, the minimal positive value of  $W^{(1,2)}_t$ is given by
\begin{equation*}
 \delta_{\min}:= \mintwo{\zeta'\ge \zeta}{\zeta'\ne \zeta} W(\zeta')-W(\zeta)=(\gl-1)\gl^{-k/2}. 
\end{equation*}
 Repeating the reasoning in \eqref{zrunk} in the asymmetric case, we obtain that 
 \begin{equation*}
     \bbP[\tau>t]\le \frac{     \bbE\left[W^{(1,2)}_t\right]}{\delta_{\min}}
 \end{equation*}
Now from \eqref{leqasymme}, $v^{(1,2)}(t,i):= \bbE\left[V(h^{(2)}_t,i)-V(h^{(1)}_t,i)\right]$ satisfies the assumption Lemma \ref{contrax2}.
We obtain, using Cauchy-Schwarz inequality, that 
\begin{equation*}
 \bbE\left[W^{(1,2)}_t\right]\le  \sum_{i=1}^{N-1} v^{(1,2)}(t,i) \le e^{-\varrho t} \left( (N-1)   \sum_{i=1}^{N-1} v^{(1,2)}(0,i)^2 \right)^{1/2}. 
\end{equation*}
Now considering that the maximal possible height-difference is $2k$ and that the maximal possible value of 
$\zeta(i)$ is always smaller than $N-k$ we have for every $i\in \lint 1,N-1\rint$
$$v^{(1,2)}(0,i)\le \max_{\zeta'\in \Xi_{N,k}} \gl^{\frac{\zeta'(i)}2}(1-\gl^{-k}) 
\le k(\gl-1)\gl^{\frac{N-k}{2}-1}.$$
Setting $\delta_{\max}:= (\gl-1)\gl^{\frac{N-k}{2}-1}$, we obtain
that $\sum_{i=1}^{N-1} v^{(1,2)}(0,i)^2\le \delta_{\max}^2 (N-1)k^2$ so that 
\begin{equation*}
 \bbP[\tau>t]\le \frac{\delta_{\max}}{\delta_{\min}} (N-1)k e^{-\varrho t}
\end{equation*}
which is the desired result.
 \end{proof}
 
 \begin{rem}\label{andtheinter}
  Note that the argument exposed in the section can also be used without changes for the interchange process.
  Indeed the correspondences exposed in Section \ref{sec:corre} allows to associate, to the dynamic 
  $\sigma_t$, $N-1$ corner-flip dynamics $(h^{(k)}_t)$, $k=1,\dots N$, defined by 
  $$ h^{(k)}_t= h\circ \xi^{(k)} \circ \sigma_t $$
 where the transformations $h$ and $\xi^{(k)}$ are those of Section \ref{sec:corre}.
The observation \eqref{thatsall} guarantees that two dynamics $\sigma^{(1)}_t$ and
$\sigma^{(2)}_t$ are coupled when all the corresponding corner-flip dynamics are coupled, 
so that the analog of Proposition \ref{monocouple} is valid for the interchange process on the segement, 
with the factor $k(N-1)$ replaced by $(N-1)^3$. The reader can refer to \cite[Section 3]{Wil04}
and \cite[Section 3.4]{LL19} for more details in the symmetric and asymmetric cases respectively.
 \end{rem}

\section{Shortcomings and possible improvements of the reasoning above}
\label{sec:tecx}
\subsection{For symmetric dynamics}

As mentioned in Remark \ref{dizcuz}, 
the upper-bound on the SSEP mixing time 
is suboptimal, off by a factor $4$ in the case when $k$ and $N-k$ are of order $N$.
There are two separate reasons for which the method does 
not yield an optimal result, each being accountable for a multiplicative
factor $2$.
To illustrate this, let us mention 
 \cite[Section 8]{Wil04}, where it 
 is proved that for the monotone coupling 
 inherited from the graphical construction 
 (described in Section \ref{sec:order})
 the coupling time $\tau_1$ in \eqref{tau1tau2} is of order 
$\frac{2}{\pi^2} N^2 \log k$.
This results shows that not only the method above is off by a factor two
to estimate the coupling time, but also, when comparing it to
Theorem \ref{thm:mixsep}, that 
this coupling time  itself does allow for a sharp estimate on the 
mixing time.
This means that in order to improve the bound 
on the mixing time, we have to design a monotone coupling 
that makes the value of the coupling time $\tau$ as small as possible.

This becomes particularly obvious when the Random Walk on the simplex 
is considered (recall Section \ref{sec:semidis}). If one considers
the monotone grand coupling based on
the graphical construction presented in Section \ref{sec:semidis},
then trajectories starting with different initial conditions
\textit{never} coalesce ($\tau=\infty$ almost surely). Hence for this model, 
there can be no equivalent of Proposition \ref{monocouple} : Any non trivial estimate on $\tau$ must rely on specific features of the coupling beyond monotonicity. 

\medskip

In \cite{CLL,CLL21,LL20,Lac16,Lac162,Lac17}, refinements have been performed  in order to obtain optimal estimates on the mixing time.
This first one is the introduction of a coupling that is aimed at minimizing the 
coalescence time. The basic idea for the discrete model is to make
the corner-flips performed by $h^{(1)}_t$ and $h^{(2)}_t$ less 
synchronized 
while preserving monotonicity so that the quantity 
$$A(t):=\sum^{N-1}_{i=1} \left(h^{(2)}_t- h^{(1)}_t\right),$$
which is an integer-valued supermartingale, hits zero faster.
Roughly speaking, this is achieved 
by having, at any given time, independent corner flips 
for coordinates at which $h^{(2)}_t(x)>h^{(1)}_t(x)$, and synchronized corner flips 
for coordinates at which $h^{(2)}_t(x)=h^{(1)}_t(x)$ (the couplings used in 
the continuous setup in \cite{CLL,CLL21} are based 
on an analogous intuition).
The second key improvement is  to use diffusion estimates in 
order to estimate the time when $A(t)$ hits $0$, instead of 
relying on Markov's inequality. For the corner-flip dynamics, $A(t)$ 
is a time changed random walk on $\bbZ_+$, and the hitting time
of $0$ can be precisely estimated if one has some control over 
its jump rate (see \cite{Lac16,Lac162,Lac17}).
This idea was considerably improved in \cite{CLL,CLL21,LL20} where we
need to estimate the hitting time of zero of a supermartingale
which is not integer valued. 
The improvements comes from reasoning  in terms of  martingale brackets 
instead of jump rate.

\subsection{For asymmetric dynamics}

Remark \ref{dizcuz} also underlines that the result of the previous section is also suboptimal in the asymmetric case.
The reason for this is that the quantity $W^{(2,1)}_t$ considered
in \eqref{defw12}  is typically much smaller than its average (by a factor which is exponential in $N$). 
Since this quantity has very wild fluctuation, it is not possible to apply to it the same technique as in the symmetric case.
The proof of  Theorem \ref{thm:mixasep} presented in   \cite{LL19} relies  on two key ingredients:
\begin{itemize}
 \item [(A)] Hydrodynamic limits.
\item [(B)] The control of particle speed when the density is vanishing.
\end{itemize}
 Hydrodynamic limits are an extensively studied topic for particle systems (see \cite{KLbook}). The hydrodynamic limit of a system is the limit obtained for the evolution of the particle density after rescaling time and space. It usually takes the form of the solution to partial differential equation.
 In the case of the asymmetric exclusion process, 
 is has been established (see \cite{Reza91} 
 where the result is proved in a much broader context) 
 that the hydrodynamic limit - after rescaling time and space by $N$ - is the solution 
 of the equation 
 \begin{equation}\label{struud}
 \partial_t \rho= (2p-1)\partial_x\left[\rho(1-\rho)\right].
 \end{equation}
 More precisely, for the exclusion on the segment,
 we have to consider some specific notion  solution and boundary conditions (see \cite[Section 5]{LL19} for details).
 In this context, given any initial condition $\rho_0$,  which satisfies 
 $$\forall x\in [0,1], \  0\le \rho_0(x)\le 1
 \quad \text{ and } \int_{[0,1]}\rho(x)=\alpha$$ 
  \eqref{struud} has a unique solution which 
  stabilizes to the fixed point $\ind_{[1-\alpha,1]}$ after a time 
 $\frac{(\sqrt{\alpha}+\sqrt{1-\alpha})^2}{2p-1}$, indicating that at time  $\frac{(\sqrt{\alpha}+\sqrt{1-\alpha})^2 N}{2p-1}$ the system is macroscopically at equilibrium.
 
 \medskip
 
 What remains to check afterwards is that around that time the system 
 is also at equilibrium in the total variation sense,
 which is \textit{a priori} a much finer statement.
The important point is to verify that the position of
the leftmost particle and rightmost empty site match 
 the indication given by the macroscopic profile 
 (that is, are both  $(1-\alpha)N+o(N)$ ), and 
 this is where the point $(B)$ comes into play 
 (we refer to \cite[Section 6]{LL19} for more details).
 
 \medskip
 
Once we have proved that that both the density of particle
 and the position of leftmost particle/rightmost empty site 
 have reached there equilibrium, we still have not proved that the system is at equilibrium.
 However this information implies that with the notation of Section 
 \ref{proofsec}, when 
 $t=t_{\alpha,N}:=\frac{(\sqrt{\alpha}+\sqrt{1-\alpha})^2}{2p-1}$ 
 we have $W^{(1,2)}_t= \exp(o(N))\delta_{\min}$.
 Hence we can use, as a third step of our reasoning, the contraction estimate of Lemma \ref{contrax2} 
 to show that coupling must occurs shortly after time $t_{\alpha,N}$.

 \section{Acknowledgements*}
 The author is gratefull to P.\ Caputo and C.\ Labbé for feedback on the manuscript.
 The author acknowledges support from a productivity grant from CNPq and from a 
 JCNE grant  from FAPERj.

\bibliographystyle{emss}
\bibliography{biblio.bib}

\begin{thebibliography}{10}
\providecommand{\url}[1]{\texttt{#1}}
\providecommand{\urlprefix}{URL }
\providecommand{\eprint}[2][]{\url{#2}}

\bibitem{BaLi18}
J.~Baik and Z.~Liu, Fluctuations of tasep on a ring in relaxation time scale.
  \emph{Communications on Pure and Applied Mathematics} \textbf{71} (2018),
  no.~4, 747--813.

\bibitem{Benjamini}
I.~Benjamini, N.~Berger, C.~Hoffman, and E.~Mossel, Mixing times of the biased
  card shuffling and the asymmetric exclusion process. \emph{Trans. Amer. Math.
  Soc.} \textbf{357} (2005), no.~8, 3013--3029 (electronic). \MR{2135733}

\bibitem{BuNe20}
A.~M. {Bufetov} and P.~{Nejjar}, {Cutoff profile of ASEP on a segment}.
  \emph{arXiv e-prints}  (2020), arXiv:2012.14924.

\bibitem{CLL}
P.~Caputo, C.~Labb\'{e}, and H.~Lacoin, Mixing time of the adjacent walk on the
  simplex. \emph{Ann. Probab.} \textbf{48} (2020), no.~5, 2449--2493.
  \MR{4152648}

\bibitem{CLL21}
P.~{Caputo}, C.~{Labb{\'e}}, and H.~{Lacoin}, {Spectral gap and cutoff
  phenomenon for the Gibbs sampler of $\nabla\varphi$ interfaces with convex
  potential}. \emph{Ann. Inst. Henri Poincar{\'e}, Probab. Stat.}  (to appear).

\bibitem{CLMST12}
P.~Caputo, H.~Lacoin, F.~Martinelli, F.~Simenhaus, and F.~L. Toninelli, Polymer
  dynamics in the depinned phase: metastability with logarithmic barriers.
  \emph{Probab. Theory Related Fields} \textbf{153} (2012), no. 3-4, 587--641.
  \MR{2948687}

\bibitem{CLR10}
P.~Caputo, T.~M. Liggett, and T.~Richthammer, Proof of {A}ldous' spectral gap
  conjecture. \emph{J. Amer. Math. Soc.} \textbf{23} (2010), no.~3, 831--851.
  \MR{2629990}

\bibitem{CLMST14}
P.~Caputo, E.~Lubetzky, F.~Martinelli, A.~Sly, and F.~L. Toninelli, Dynamics of
  {$(2+1)$}-dimensional {SOS} surfaces above a wall: slow mixing induced by
  entropic repulsion. \emph{Ann. Probab.} \textbf{42} (2014), no.~4,
  1516--1589. \MR{3262485}

\bibitem{CMST}
P.~Caputo, F.~Martinelli, F.~Simenhaus, and F.~L. Toninelli, ``{Z}ero''
  temperature stochastic 3{D} {I}sing model and dimer covering fluctuations: a
  first step towards interface mean curvature motion. \emph{Comm. Pure Appl.
  Math.} \textbf{64} (2011), no.~6, 778--831. \MR{2663712}

\bibitem{CMT08}
P.~Caputo, F.~Martinelli, and F.~L. Toninelli, On the approach to equilibrium
  for a polymer with adsorption and repulsion. \emph{Electron. J. Probab.}
  \textbf{13} (2008), no. 10, 213--258. \MR{2386733}

\bibitem{CMT12}
P.~Caputo, F.~Martinelli, and F.~L. Toninelli, Mixing times of monotone
  surfaces and {SOS} interfaces: a mean curvature approach. \emph{Comm. Math.
  Phys.} \textbf{311} (2012), no.~1, 157--189. \MR{2892467}

\bibitem{DiaconisPNAS}
P.~Diaconis, The cutoff phenomenon in finite markov chains. \emph{Proceedings
  of the National Academy of Sciences} \textbf{93} (1996), no.~4, 1659--1664.

\bibitem{GaGh21}
S.~{Ganguly} and R.~{Gheissari}, {Cutoff for the Glauber dynamics of the
  lattice free field}. \emph{arXiv e-prints}  (2021), arXiv:2108.07791.

\bibitem{GK13}
N.~Gantert and T.~Kochler, Cutoff and mixing time for transient random walks in
  random environments. \emph{ALEA Lat. Am. J. Probab. Math. Stat.} \textbf{10}
  (2013), no.~1, 449--484. \MR{3083933}

\bibitem{GNS2020}
N.~{Gantert}, E.~{Nestoridi}, and D.~{Schmid}, {Mixing times for the simple
  exclusion process with open boundaries}. \emph{arXiv e-prints}  (2020),
  arXiv:2003.03781.

\bibitem{GRSP2020}
S.~Greenberg, D.~Randall, and A.~P. Streib, Sampling biased monotonic surfaces
  using exponential metrics. \emph{Combin. Probab. Comput.} \textbf{29} (2020),
  no.~5, 672--697. \MR{4152566}

\bibitem{GART88}
J.~Gärtner, Convergence towards burger's equation and propagation of chaos for
  weakly asymmetric exclusion processes. \emph{Stochastic Processes and their
  Applications} \textbf{27} (1988), 233--260.

\bibitem{HS21}
J.~Hermon and J.~Salez, The interchange process on high-dimensional products.
  \emph{Ann. Appl. Probab.} \textbf{31} (2021), no.~1, 84--98. \MR{4254474}

\bibitem{HoJi17}
B.~Hough and Y.~Jiang, Cut-off phenomenon in the uniform plane {K}ac walk.
  \emph{Ann. Probab.} \textbf{45} (2017), no.~4, 2248--2308. \MR{3693962}

\bibitem{KKS75}
H.~Kesten, M.~V. Kozlov, and F.~Spitzer, A limit law for random walk in a
  random environment. \emph{Compositio Math.} \textbf{30} (1975), 145--168.
  \MR{380998}

\bibitem{KLbook}
C.~Kipnis and C.~Landim, \emph{Scaling limits of interacting particle systems}.
  Grundlehren der Mathematischen Wissenschaften [Fundamental Principles of
  Mathematical Sciences] 320, Springer-Verlag, Berlin, 1999. \MR{1707314}

\bibitem{LL19}
C.~Labb\'{e} and H.~Lacoin, Cutoff phenomenon for the asymmetric simple
  exclusion process and the biased card shuffling. \emph{Ann. Probab.}
  \textbf{47} (2019), no.~3, 1541--1586. \MR{3945753}

\bibitem{LL20}
C.~Labb\'{e} and H.~Lacoin, Mixing time and cutoff for the weakly asymmetric
  simple exclusion process. \emph{Ann. Appl. Probab.} \textbf{30} (2020),
  no.~4, 1847--1883. \MR{4132639}

\bibitem{Lac162}
H.~Lacoin, The cutoff profile for the simple exclusion process on the circle.
  \emph{Ann. Probab.} \textbf{44} (2016), no.~5, 3399--3430. \MR{3551201}

\bibitem{Lac16}
H.~Lacoin, Mixing time and cutoff for the adjacent transposition shuffle and
  the simple exclusion. \emph{Ann. Probab.} \textbf{44} (2016), no.~2,
  1426--1487. \MR{3474475}

\bibitem{Lac17}
H.~Lacoin, The simple exclusion process on the circle has a diffusive cutoff
  window. \emph{Ann. Inst. Henri Poincar\'{e} Probab. Stat.} \textbf{53}
  (2017), no.~3, 1402--1437. \MR{3689972}

\bibitem{LaTe15}
H.~Lacoin and A.~Teixeira, A mathematical perspective on metastable wetting.
  \emph{Electron. J. Probab.} \textbf{20} (2015), no. 17, 23. \MR{3317159}

\bibitem{LY20}
H.~{Lacoin} and S.~{Yang}, {Metastability for expanding bubbles on a sticky
  substrate}. \emph{arXiv e-prints}  (2020), arXiv:2007.07832.

\bibitem{LY21}
H.~{Lacoin} and S.~{Yang}, {Mixing time for the asymmetric simple exclusion
  process in a random environment}. \emph{arXiv e-prints}  (2021),
  arXiv:2102.02606.

\bibitem{LePe16}
D.~A. Levin and Y.~Peres, Mixing of the exclusion process with small bias.
  \emph{J. Stat. Phys.} \textbf{165} (2016), no.~6, 1036--1050. \MR{3575636}

\bibitem{MCMT}
D.~A. Levin and Y.~Peres, \emph{Markov chains and mixing times (second
  edition)}. MBK, American Mathematical Society, 2017.

\bibitem{Liggett}
T.~M. Liggett, \emph{Interacting particle systems}. Classics in mathematics,
  Springer, 2005.

\bibitem{Lyonsbook}
R.~Lyons and Y.~Peres, \emph{Probability on trees and networks}. Cambridge
  Series in Statistical and Probabilistic Mathematics 42, Cambridge University
  Press, New York, 2016. \MR{3616205}

\bibitem{Morris06}
B.~Morris, The mixing time for simple exclusion. \emph{Ann. Appl. Probab.}
  \textbf{16} (2006), no.~2, 615--635. \MR{2244427}

\bibitem{Oli13}
R.~I. Oliveira, Mixing of the symmetric exclusion processes in terms of the
  corresponding single-particle random walk. \emph{Ann. Probab.} \textbf{41}
  (2013), no.~2, 871--913. \MR{3077529}

\bibitem{Qua92}
J.~Quastel, Diffusion of color in the simple exclusion process. \emph{Comm.
  Pure Appl. Math.} \textbf{45} (1992), no.~6, 623--679. \MR{1162368}

\bibitem{RW05b}
D.~Randall and P.~Winkler, Mixing points on a circle. In \emph{Approximation,
  randomization and combinatorial optimization. algorithms and techniques}, pp.
  426--435, Springer, 2005.

\bibitem{RW05}
D.~Randall and P.~Winkler, Mixing points on an interval. In \emph{Proceedings
  of the second workshop on analytic algorithms and combinatorics, vancouver,
  2005}, pp. 216--221, 2005.

\bibitem{Revuz84}
D.~Revuz, \emph{Markov chains}. Second edn., North-Holland Mathematical Library
  11, North-Holland Publishing Co., Amsterdam, 1984. \MR{758799}

\bibitem{Reza91}
F.~Rezakhanlou, Hydrodynamic limit for attractive particle systems on
  {${\mathbf Z}^d$}. \emph{Comm. Math. Phys.} \textbf{140} (1991), no.~3,
  417--448. \MR{1130693}

\bibitem{Rowi2000}
L.~C.~G. Rogers and D.~Williams, \emph{Diffusions, markov processes, and
  martingales}. 2 edn., Cambridge Mathematical Library 1, Cambridge University
  Press, 2000.

\bibitem{Schmid19}
D.~Schmid, Mixing times for the simple exclusion process in ballistic random
  environment. \emph{Electron. J. Probab.} \textbf{24} (2019), Paper No. 22,
  25. \MR{3933201}

\bibitem{Schmid21}
D.~{Schmid}, {Mixing times for the TASEP in the maximal current phase}.
  \emph{arXiv e-prints}  (2021), arXiv:2104.12745.

\bibitem{Sol75}
F.~Solomon, Random walks in a random environment. \emph{Ann. Probability}
  \textbf{3} (1975), 1--31. \MR{362503}

\bibitem{Strook2014}
D.~W. Stroock, \emph{An introduction to markov processes}. 2 edn., Graduate
  Texts in Mathematics 230, Springer-Verlag Berlin Heidelberg, 2014.

\bibitem{Wil04}
D.~B. Wilson, Mixing times of {L}ozenge tiling and card shuffling {M}arkov
  chains. \emph{Ann. Appl. Probab.} \textbf{14} (2004), no.~1, 274--325.
  \MR{2023023}

\bibitem{Yang21}
S.~Yang, Cutoff for polymer pinning dynamics in the repulsive phase. \emph{Ann.
  Inst. Henri Poincar\'{e} Probab. Stat.} \textbf{57} (2021), no.~3,
  1306--1335. \MR{4291457}

\bibitem{Yau97}
H.-T. Yau, Logarithmic {S}obolev inequality for generalized simple exclusion
  processes. \emph{Probab. Theory Related Fields} \textbf{109} (1997), no.~4,
  507--538. \MR{1483598}

\end{thebibliography}

\end{document}